\documentclass[reqno, 12pt]{amsart}

\usepackage{amsfonts,amsthm,amsmath,amssymb,amscd,mathrsfs}
\usepackage[colorlinks=true,linkcolor=blue,citecolor=red,urlcolor=black]{hyperref}
\usepackage{graphics}
\usepackage{indentfirst}
\usepackage{cite}
\usepackage{latexsym}
\usepackage[dvips]{epsfig}
\usepackage{color}
\usepackage{lmodern}
\usepackage[alphabetic,nobysame]{amsrefs}
 \def\MR#1{}

\usepackage{soul}

\newtheorem{theorem}{Theorem}[section]

\newtheorem{remark}{Remark}[section]

\newtheorem{lemma}[theorem]{Lemma}
\newtheorem{proposition}[theorem]{Proposition}

\numberwithin{equation}{section}

\renewcommand{\div}{{\rm div \thinspace }}

\newcommand{\bt}{\begin{theorem}}
	\newcommand{\bl}{\begin{lemma}}
		\newcommand{\el}{\end{lemma}}
	\newcommand{\et}{\end{theorem}}

\newcommand{\curl}{{\rm curl}~}

\newcommand{\bR}{\mathbb{R}}

\newcommand{\bBV}{\boldsymbol{V}}
\newcommand{\Bn}{{\boldsymbol{n}}}

\newcommand{\Bu}{{\boldsymbol{u}}}

\newcommand{\supp}{\text{supp}}

\newcommand{\Bo}{{\boldsymbol{\omega}}}

\newcommand{\Bw}{{\boldsymbol{w}}}

\begin{document}
\title
         [Saint Venant Estimates]
	{Saint-Venant Estimates and Liouville-Type Theorems for the stationary Navier-Stokes equation in $\mathbb{R}^3$}

\author[J. Bang]{Jeaheang Bang}
	\address[J. Bang]{Department of Mathematics,
Purdue University,
150 N. University Street,
West Lafayette, IN 47907, USA}
	\email{bang30@purdue.edu}

	\author[Z. Yang]{Zhuolun Yang}
	\address[Z. Yang]{
                Division of Applied Mathematics, Brown University, 182 George Street, Providence, RI 02912, USA
                }
	\email{zhuolun\_yang@brown.edu}

 \thanks{Z. Yang was partially supported by the AMS-Simons Travel Grant.}


\begin{abstract}
    	We prove two Liouville type theorems for the stationary Navier-Stokes equations in $\mathbb{R}^3$  under some assumptions on 1) the growth of the $L^s$ mean oscillation of a potential function of the velocity field, or 2) the relative decay of the head pressure and the square of the velocity field at infinity. The main idea is to use Saint-Venant type estimates to characterize the growth of Dirichlet energy of nontrivial solutions. These assumptions are  weaker than those previously known of a similar nature. 
\end{abstract}

 \maketitle

 \section{Introduction and Main result}

Let us consider the stationary Navier-Stokes equations in $\mathbb{R}^3$:
    \begin{align} \label{SNS}
        -\Delta \Bu 
        + \Bu \cdot \nabla \Bu + \nabla p =0, 
        \quad \div \Bu=0 \quad \text{in }\mathbb{R}^3.
    \end{align}
Proving the triviality of $\Bu$ with some assumptions is called a Liouville-type problem. Following the work of Leray \cite{Leray33}, it is customary to assume
    \begin{align} 
    \label{FE}
        \nabla \Bu \in L^2 (\mathbb{R}^3), 
            \\
    \label{BC}
        \lim_{|x|\to\infty} \Bu(x) =0.
    \end{align}
Proving the triviality of $\Bu$ under these assumptions \eqref{FE} and \eqref{BC} is an outstanding open problem.  If we assume that a solution $\Bu$ converges to a nonzero constant vector instead of \eqref{BC}, we can show that $\Bu$ is the constant. See  \cite{Galdi11}*{Theorem X.7.2}.

For dimensions $n\neq 3$, the Liouville problem under the assumptions \eqref{FE} and \eqref{BC} has been successfully addressed; The two-dimensional case was solved by Gilbarg and Weinberger \cite{GilbargWeinberger78} about 45 years ago (which is also proved in \cite{KochNadirashviliSereginSverak09} later),
and the higher dimensional case was settled by Galdi \cite{Galdi11}*{Theorem X.9.5} by using standard energy estimates.  For the three-dimensional case, the same energy method works with the extra assumption $\Bu\in L^{9/2}(\mathbb{R}^3)$.  However, there is a substantial disparity between $L^{9/2}$ and the natural energy space $L^6$.
In terms of narrowing the gap between these two spaces, a logarithmic achievement has been made by Chae and Wolf in \cite{ChaeWolf16}: 
if a solution $\Bu$ to \eqref{SNS} satisfies
    \begin{align*}
       \int_{\mathbb{R}^3}
       \frac{|\Bu|^{\frac{9}{2}}}{  
\log\left(2+\frac{1}{|\Bu|}\right) }\thinspace dx<\infty,
    \end{align*}
then $\Bu \equiv 0.$

There are numerous partial results with various additional assumptions.  
In \cite{KozonoTerasawaWakasugi17}, Kozono, Terasawa, and Wakasugi proved the triviality of a solution $\Bu$ to \eqref{SNS}-\eqref{BC} assuming the vorticity $\boldsymbol{\omega}$ satisfies either
$\boldsymbol{\omega}=o(|x|^{-5/3})$ as $|x|\to\infty$ or $\|\boldsymbol{\omega}\|_{L^{9/5,\infty}}\leq \varepsilon \|\nabla \Bu\|_{L^2(\mathbb{R}^3)}^{2/3}$ for some small $\varepsilon>0.$ Similar results for axisymmetric solutions were shown in \cites{Wang19,Zhao19}, where the authors showed either $\Bu = O(r^{-(2/3)^+})$ or $\boldsymbol{\omega}=O(r^{-(5/3)^+})$ implies $\Bu \equiv 0$, with $r = \sqrt{x_1^2 + x_2^2}$. Notably, the 3D Liouville problem is still open even for the axisymmetric case. If one further assumes that a solution does not have a swirl component, it was first proved in \cite{KochNadirashviliSereginSverak09}, and then by a different method in \cite{KorobkovPileckasRusso15}.

Seregin \cite{Seregin16}  established a Caccioppoli-type inequality to prove the triviality of a solution $\Bu$ to \eqref{SNS} with
\begin{align*}
    \Bu \in L^6 (\mathbb{R}^3) \cap BMO^{-1} (\mathbb{R}^3),
\end{align*}
where $\Bu\in BMO^{-1}(\mathbb{R}^3)$ means $\Bu=\div \bBV$ for some anti-symmetric tensor $\bBV\in BMO(\mathbb{R}^3)$. The $BMO^{-1}$ condition was later relaxed to growth conditions on a mean oscillation of $\bBV$ in \cites{Seregin18, ChaeWolf19}. The first result of this paper further relaxes this growth condition by characterizing the growth of Dirichlet energy.

Chae in \cites{Chae14, Chae20} proved the triviality of a solution $\Bu$ to \eqref{SNS} with \eqref{BC} provided one of the following conditions hold:
\begin{align*}
\Delta \Bu \in L^{6/5}(\mathbb{R}^3)
\quad 
\text{or}
\quad
    \nabla \sqrt{|Q|}\in L^2 (\mathbb{R}^3).
\end{align*}
Here, $Q$ is the Bernoulli head pressure function, defined by
\begin{equation}\label{Def:Q}
    Q:= \frac{|\Bu|^2}{2}+ p.
\end{equation}
This quantity played an important role in numerous results on the stationary Navier-Stokes equations, such as \cites{FrehseRuzicka98, KorobkovPileckasRusso15c, LiYang22,BangGuiLiuWangXie23}.
These conditions have the same scaling as \eqref{FE}.

Thanks to \eqref{FE}, one can obtain that $p(x) \to p_0$ as $|x|\to\infty$ for some constant $p_0$ (cf. \cite{Galdi11}*{Theorem X.5.1}). As the pressure $p$ is defined up to a constant, we redefine it so that 
$p (x) \to 0$   as  $|x|\to \infty$ when assuming \eqref{FE}.
Then due to \eqref{BC},
\begin{align} \label{Qinfty}
    Q(x) \to 0 \quad \text{as }|x|\to\infty.
\end{align}
Chae in \cite{Chae21} proved the triviality of $\Bu$ with the additional assumption: 
    \begin{align}
    \label{cond213}
        \sup_{\mathbb{R}^3 }
        \frac{|\Bu|^2}{|Q|}<\infty. 
    \end{align}
This condition implies that if there exists a nontrivial solution $\Bu$ to \eqref{SNS}-\eqref{BC}, then the head pressure $Q$ should decay faster than $|\Bu|^2$. In other words, there should be some cancellation between $|\Bu|^2/2$ and $p$ near the infinity. The second result of this paper relaxes the assumption \eqref{cond213}, indicating this cancellation should be stronger.

In addition, in \cite{Tsai21}, the author established Liouville-type theorems in various domains including $\mathbb{R}^3$ by imposing an assumption on the $L^q$ norm of a solution in an annulus. The key was to find the radial dependence of the constant involved in the estimate of the Bogovskii map in the annulus.

For more results on Liouville-type theorems for the stationary Navier-Stokes equation, we refer the reader to \cites{PlechacSverak03,ChaeYoneda13, ChaeWeng16, SereginWang19, CarrilloPanZhang20, CarrilloPanZhangZhao20, Chae23, ChoChoiYang23, KozonoTerasawaWakasugi23}, and the references therein. History of this topic within a broader context is also available in \cite{Galdi11}*{Section I.2.1}.

Before stating our main results, let us introduce some notations.  Throughout this paper, we use the notation $\|f\|_{p}$ to denote the usual $L^p$ norm of $f$ in the whole space $\bR^3$, $\|f\|_{p,R}$ to denote the $L^p$ norm of $f$ in $B_R$, and $(f)_{B_R}$ to denote the average of $f$ in $B_R$.

The first result of this paper further relaxes the assumption of \cite{ChaeWolf19}*{Theorem 1.1}. The precise statement is as follows.

\begin{theorem}\label{Thm_1}
   Let $\Bu$ be a smooth solution of \eqref{SNS}. Suppose that there exist an $s \in (3,6]$ and a smooth anti-symmetric potential $\bBV \in C^\infty(\bR^3; \bR^{3\times 3})$ such that $\nabla \cdot \bBV = \Bu$, and 
    \begin{equation}\label{V_condition}
        \| \bBV - (\bBV)_{B_R}\|_{s,R} \lesssim R^{\frac{s+6}{3s}} (\log R)^{\frac{s-3}{3s}} \quad \mbox{for}~R > 2.
    \end{equation}
    Then $\Bu \equiv 0$.
\end{theorem}

\begin{remark} 
Theorem \ref{Thm_1} is stated for $s\in (3,6]$, but 
it can generalized to $s\in(3,\infty)$ by replacing \eqref{V_condition} with
    \begin{equation*}
        \| \bBV - (\bBV)_{B_R}\|_{s,R} \lesssim R^{\min\{\frac{s+6}{3s},\frac{s+18}{6s}\}} (\log R)^{\min\{\frac{s-3}{3s},\frac16\}} \quad \mbox{for}~R > 2.
    \end{equation*}
Indeed, if $s>6$, by H\"older's inequality,
    \begin{align*}
    \|\bBV - (\bBV)_{B_R} \|_{6,R}
    \lesssim
    \|\bBV - (\bBV)_{B_R}\|_{s,R} \,
    R^{ \frac{s-6}{2s} }
    \lesssim
    R^{\frac23} (\log R)^{\frac16},
    \end{align*}
which is \eqref{V_condition} with $s=6$.  For $s = 3$, we do not have a control of the local Dirichlet energy (see Lemma \ref{Lemma_energy_upperbound} below). In this case, by additionally assuming a growth condition on $\|\nabla\Bu\|_{2,R}$, one can still follow the proof to relax the main condition of \cite{Seregin16}*{Theorem 1.1}.
\end{remark}

\begin{remark}
    Compared to the results in \cite{ChaeWolf19}, we further assume the anti-symmetry of $\bBV$ but relax the growth condition on the oscillation of $\bBV$ by a logarithmic factor. We would like to point out that for any divergence free vector field $\Bu$, there always exists an anti-symmetric potential $\bBV$ such that $\Bu = \div \bBV$. Regularity assumptions on such anti-symmetric potentials are commonly made in the literature on fluid equations. See, e.g., \cites{SSSZ,Seregin16}. 
\end{remark} 

Our second result of this paper improves \cite{Chae21}*{Theorem 1.1}. Recall that under the conditions \eqref{FE} and \eqref{BC}, we may assume without loss of generality that the head pressure $Q$ satisfies \eqref{Qinfty}.

\begin{theorem}\label{Thm_2}
Let $\Bu$ be a smooth solution to \eqref{SNS}-\eqref{BC}. 
Assume that 
\begin{equation}\label{cond2}
    \limsup_{|x| \to \infty}
	\frac{|\Bu |^2}
 {|Q||\log |Q||} <\infty,
\end{equation}
where $Q$ is defined in \eqref{Def:Q} and satisfies \eqref{Qinfty}. Then $\Bu \equiv 0$.
\end{theorem}

\begin{remark}
    The $(\log R)$ factor in \eqref{V_condition} and $|\log |Q||$ in \eqref{cond2} can be further relaxed to $a(R)$ and $a(|Q|)$, respectively, where $$
    a(s) = \prod_{k=1}^m l^k(s)
    $$
    for any $m \in \mathbb{N}$, $l^k$ is the composition of $|\log|$ function k times. For example, when $m=2$, $a(s)=\big|\log |\log s|\big| \times |\log s|$. We rely on the facts that $1/(sa(s))$ is not integrable at $s = \infty$ in the proof of Theorem \ref{Thm_1}, and it is not integrable at $s=0$ in the proof of Theorem \ref{Thm_2}.
\end{remark}

A key ingredient of our proof is a type of estimates that was used in \cite{Toupin65} to first rigorously established Saint-Venant's principle for the linear systems of elasticity. Such estimates were later referred to as Saint-Venant type estimates (SV estimates) by Ladyzhenskaya and Solonnikov in \cite{LadyzhenskayaSolonnikov80}. To elaborate on the idea, one first uses a specialized cut-off function to obtain an energy estimate, which leads to a (ordinary) differential inequality of the energy over a growing bounded region. By studying the differential inequality, one can characterize the behavior of the energy of a nontrivial solution.

In the studies of the stationary Navier-Stokes equations, SV estimates were applied to establish Liouville-type theorems in a cylinder-like domain in \cite{LadyzhenskayaSolonnikov80}, and were later extended to the slab domain $\bR^2 \times (0,1)$ in \cite{BangGuiWangXie22}. See also \cite{KozonoTerasawaWakasugi23b}. We emphasize that it is crucial for the domain to be bounded in at least one direction for the use of the Poincar\'e inequality to obtain the desired differential inequality in the literature above.

Let us provide a brief overview of our proof strategies. To prove Theorem \ref{Thm_1} by using SV estimates, we use the assumption $\Bu=\div \bBV$ to introduce extra derivatives to control the nonlinear term; and a noteworthy element is the use of the Bogovskii map on an annular region to handle the pressure term, with additional attention on the constant dependence on the radii. For Theorem \ref{Thm_2}, we define a special cut-off function through level sets of $Q$ to handle both the nonlinear and pressure terms. Then,  the (elliptic) equation of the head pressure, \eqref{hpeq}, is used to control additional factors due to the special cut-off function. In the end, we derive characterizations of the local Dirichlet energy of nontrivial solutions through SV estimates, leading to contradictions.

The rest of paper is organized as follows. In Section \ref{1thm}, we provide a proof of Theorem \ref{Thm_1}, and Section \ref{2thm} is devoted to a proof of Theorem \ref{Thm_2}.

\section{Proof of Theorem \ref{Thm_1}} \label{1thm}

Throughout this section,  we denote $\overline \bBV: = \bBV - (\bBV)_{B_R}$ for convenience. We will prove Theorem \ref{Thm_1} through the following lemmas. The first two lemmas follow essentially from \cite{ChaeWolf19}*{Lemma 2.1 and Lemma 2.2}.

\begin{lemma}\label{Lemma_u_L2}
    Let $0 < \rho < R$, $s \ge 2$, and $\eta_1$ be a cut-off function such that $\eta_1 = 1$ in $B_{(R+\rho)/2}$, $\eta_1 = 0$ in $B_R^\complement$,  $|\nabla \eta_1| \lesssim (R-\rho)^{-1}$. Then
    \begin{equation}
        \label{u_L2_estimate}
        \|\Bu \eta_1\|_{2,R} \lesssim R^{\frac{3(s-2)}{4s}}\|\overline \bBV\|_{s,R}^{\frac{1}{2}} \|\nabla \Bu\|_{2,R}^{\frac{1}{2}} +  R^{\frac{3(s-2)}{2s}} (R-\rho)^{-1} \|\overline \bBV\|_{s,R}.
    \end{equation}
\end{lemma}

\begin{proof}
    Using $\Bu = \nabla \cdot \overline\bBV$, integration by parts and H\"older's inequality, we have
    \begin{align*}
         \|\Bu \eta_1\|_{2,R}^2 = \int_{B_R} |\Bu|^2 \eta_1^2 &= \int_{B_R} \partial_i \overline V_{ij} u_j \eta_1^2\\
         &= - \int_{B_R} \overline V_{ij} \partial_i u_j \eta_1^2 - 2 \int_{B_R} \overline V_{ij} u_j \partial_i \eta_1 \eta_1\\
         &\lesssim \| \overline \bBV \|_{2,R} \| \nabla \Bu \|_{2,R} + (R-\rho)^{-1} \| \overline \bBV \|_{2,R}  \|\Bu \eta_1\|_{2,R}. 
    \end{align*}
    Therefore, by Young's inequality and H\"older's inequality, we have
    \begin{align*}
        \|\Bu \eta_1\|_{2,R}^2  &\lesssim \| \overline \bBV \|_{2,R} \| \nabla \Bu \|_{2,R} + (R-\rho)^{-2} \| \overline \bBV \|_{2,R}^2\\
        &\lesssim R^{\frac{3(s-2)}{2s}} \| \overline \bBV \|_{s,R} \| \nabla \Bu \|_{2,R} 
        + R^{\frac{3(s-2)}{s}} (R-\rho)^{-2} \| \overline \bBV \|_{s,R}^2.
    \end{align*}
    The lemma is proved.
\end{proof}

\begin{lemma}\label{Lemma_u_Ls}
    Let $s \ge 2$, and $\eta_1$ be the cut-off function as in Lemma \ref{Lemma_u_L2}. Then
    \begin{equation}
        \label{u_Ls_estimate}
        \|\Bu \eta_1\|_{\frac{4s}{s+2},R}
        \lesssim
        \|\overline{\bBV}\|_{s,R}^{\frac{1}{2}}
        \|\nabla \Bu\|_{2,R}^{\frac12}
        +
        (R-\rho)^{-1}
        R^{\frac{3s-6}{4s}}
        \|\overline{\bBV}\|_{s,R}.
    \end{equation}
\end{lemma}
\begin{proof}
     Using $\Bu = \nabla \cdot \overline\bBV$, integration by parts, H\"older's inequality, and Young's inequality, we have
     \begin{align*}
         \|\Bu \eta_1\|_{\frac{4s}{s+2},R}^{\frac{4s}{s+2}} 
          = 
          \int_{B_R} |\Bu|^{\frac{4s}{s+2}} \eta_1^{ \frac{4s}{s+2} } 
          &= \int_{B_R} \partial_i \overline V_{ij} u_j |\Bu|^{\frac{2s-4}{s+2}} \eta_1^{\frac{4s}{s+2}}\\
           &= - \int_{B_R} \overline V_{ij} \partial_i (u_j|\Bu|^{\frac{2s-4}{s+2}}) \eta_1^{\frac{4s}{s+2}} - \frac{4s}{s+2} \int_{B_R} \overline V_{ij} u_j |\Bu|^{\frac{2s-4}{s+2}}\partial_i \eta_1 \eta_1^{\frac{3s-2}{s+2}}\\
           &\lesssim \| \overline \bBV \|_{s,R} \| \nabla \Bu \|_{2,R}\|\Bu \eta_1\|_{\frac{4s}{s+2},R}^{\frac{2s-4}{s+2}} + (R-\rho)^{-1} \| \overline \bBV \|_{\frac{4s}{s+2},R}\|\Bu \eta_1\|_{\frac{4s}{s+2},R}^{\frac{3s-2}{s+2}}\\
           &\lesssim \| \overline \bBV \|_{s,R}^{\frac{2s}{s+2}} \| \nabla \Bu \|_{2,R}^{\frac{2s}{s+2}} + R^{\frac{3s-6}{s+2}}(R-\rho)^{-\frac{4s}{s+2}} \| \overline \bBV \|_{s,R}^{\frac{4s}{s+2}}.
     \end{align*}
     The lemma is proved.
\end{proof}

Next, we use the equation \eqref{SNS} along with the main condition \eqref{V_condition} to derive an energy estimate.

\begin{proposition}\label{Prop_energy}
     Let $2 < \rho < R$, $s \in [3,6]$, $\bBV$ satisfy \eqref{V_condition}, and $\eta_2$ be a cut-off function such that $\eta_2 = 1$ in $B_{\rho}$, $\eta_2 = 0$ in $B_{(R+\rho)/2}^\complement$,  $|\nabla \eta_2| \lesssim (R-\rho)^{-1}$. Then
     \begin{equation}
         \label{energy}
         \begin{aligned}
          \| \nabla \Bu \, \sqrt{\eta_2} \|_{2,R}^2 \lesssim R^{10}(R- \rho)^{-10} (\log R)^{\frac{s-3}{s+6}} \left( \int_{B_{(R+\rho)/2} \setminus B_\rho} |\nabla \Bu|^2 \right)^{\frac{1}{2}} \left( \| \nabla \Bu \|_{2,R}^{\frac{12-s}{s+6}}+ 1 \right). 
         \end{aligned}
     \end{equation}
\end{proposition}

\begin{proof}
    First, using the Bogovskii map on an annulus, we construct a function to take care of the pressure term. For the construction of the Bogovskii map, see \cite{Galdi11}*{Theorem III.3.5} or \cite{Tsai18}*{Section 2.8}.
    Note that 
    \begin{align*}
        \int_{B_{(R+\rho)/2} \setminus B_\rho} \Bu \cdot \nabla \eta_2
        \thinspace dx
        =-\int _{\partial B_{\rho}} \Bu\cdot \Bn \thinspace d\sigma =0.
    \end{align*}
    Using \cite{Tsai21}*{Lemma 3}, for any $q\in (1,\infty)$, there exists a constant $c=c(q)>0$, independent of $R$ and $\rho$, and a function $\Bw\in W^{1,q}_0 (B_{(R+\rho)/2} \setminus B_\rho)$ such that $\div \Bw =  \Bu \cdot \nabla \eta_2$ and
    \begin{align}\label{bogo_estimate}
        \int_{B_{(R+\rho)/2} \setminus B_\rho} |\nabla \Bw|^q \thinspace dx
        \leq
        c\, \left( \frac{\rho}{R-\rho}\right)^q \int_{B_{(R+\rho)/2} \setminus B_\rho} | \Bu \cdot \nabla \eta_2|^q \thinspace dx.
    \end{align}
    We extend $\Bw=0$ in $B_\rho$. Now we multiply the equation \eqref{SNS} by $\Bu \eta_2 - \Bw$ and integrate by parts in $B_R$, we will have
    \begin{align}\label{integal_split}
    \begin{aligned}
    \int |\nabla \Bu|^2 \eta_2 \, dx =& - \int \nabla \eta_2\cdot \nabla \Bu \cdot  \Bu \, dx + \int \nabla \Bu \cdot \nabla \Bw\, dx  \\ 
     &- \int (\Bu \cdot \nabla \Bu) \cdot \Bu \eta_2 \, dx + \int (\Bu \cdot \nabla \Bu) \cdot \Bw \, dx \\
     =&: I + II + III + IV.
    \end{aligned}
\end{align}
Here the pressure term vanishes since $\Bu \eta_2 - \Bw$ is divergence free. By H\"older's inequality and \eqref{u_L2_estimate}, we can estimate
\begin{equation}\label{estimate_I}
 \begin{aligned}
    |I| \lesssim& (R-\rho)^{-1}  \left( \int_{B_{(R+\rho)/2} \setminus B_\rho} |\nabla \Bu|^2 \right)^{\frac{1}{2}} \left( \int_{B_{(R+\rho)/2} \setminus B_\rho} |\Bu|^2 \right)^{\frac{1}{2}} \\
    \lesssim& (R-\rho)^{-1}  \left( \int_{B_{(R+\rho)/2} \setminus B_\rho} |\nabla \Bu|^2 \right)^{\frac{1}{2}} \left( \int_{B_R} |\Bu|^2 \eta_1^2 \right)^{\frac{1}{2}}\\
     \lesssim& (R-\rho)^{-1}  \left( \int_{B_{(R+\rho)/2} \setminus B_\rho} |\nabla \Bu|^2 \right)^{\frac{1}{2}} \left[  R^{\frac{3(s-2)}{4s}}\|\overline \bBV\|_{s,R}^{\frac{1}{2}} \|\nabla \Bu\|_{2,R}^{\frac{1}{2}} +  R^{\frac{3(s-2)}{2s}} (R-\rho)^{-1} \|\overline \bBV\|_{s,R} \right],
\end{aligned}   
\end{equation}
where $\eta_1$ is the cut-off function defined in Lemma \ref{Lemma_u_L2}. Similarly, by H\"older's inequality, \eqref{bogo_estimate} with $q = 2$, and \eqref{u_L2_estimate}, we can estimate
\begin{equation}\label{estimate_II}
 \begin{aligned}
    |II| 
    \lesssim& 
    \left( \int_{B_{(R+\rho)/2} \setminus B_\rho} |\nabla \Bu|^2 \right)^{\frac{1}{2}} \left( \int_{B_{(R+\rho)/2} \setminus B_\rho} |\nabla \Bw|^2 \right)^{\frac{1}{2}} \\
    \lesssim& 
    \rho\, (R-\rho)^{-2}  \left( \int_{B_{(R+\rho)/2} \setminus B_\rho} |\nabla \Bu|^2 \right)^{\frac{1}{2}} \left( \int_{B_{(R+\rho)/2} \setminus B_\rho} |\Bu|^2 \right)^{\frac{1}{2}}\\
     \lesssim& 
     \rho\, (R-\rho)^{-2}  \left( \int_{B_{(R+\rho)/2} \setminus B_\rho} |\nabla \Bu|^2 \right)^{\frac{1}{2}} \left[  R^{\frac{3(s-2)}{4s}}\|\overline \bBV\|_{s,R}^{\frac{1}{2}} \|\nabla \Bu\|_{2,R}^{\frac{1}{2}} +  R^{\frac{3(s-2)}{2s}} (R-\rho)^{-1} \|\overline \bBV\|_{s,R} \right].
\end{aligned}   
\end{equation}

Putting \eqref{estimate_I}, \eqref{estimate_II} together, 
    \begin{multline}
    \label{estimate_I+II}
        |I|+|II|
        \lesssim
        \\
        R(R-\rho)^{-2}  \left( \int_{B_{(R+\rho)/2} \setminus B_\rho} |\nabla \Bu|^2 \right)^{\frac{1}{2}} \left[  R^{\frac{3(s-2)}{4s}}\|\overline \bBV\|_{s,R}^{\frac{1}{2}} \|\nabla \Bu\|_{2,R}^{\frac{1}{2}} +  R^{\frac{3(s-2)}{2s}} (R-\rho)^{-1} \|\overline \bBV\|_{s,R} \right].
    \end{multline}

To estimate $III$ and $IV$, we replace $\Bu$ by $\div \overline\bBV$. Using integration by parts, we have
\begin{align*}
    III =&  -\int_{B_R} (\partial_k \overline{V}_{ki}) (\partial_i u_j) u_j \eta_2 \\
    =&  \int_{B_R} \overline{V}_{ki} \big[ (\partial_{ik} u_j) u_j \eta_2 + (\partial_i u_j) ( \partial_k u_j) \eta_2 + (\partial_i u_j) u_j  ( \partial_k\eta_2) \big] \\ 
    =& \int_{B_{(R+\rho)/2} \setminus B_\rho} \overline{V}_{ki} (\partial_i u_j) u_j  ( \partial_k\eta_2),
\end{align*}
where in the last line, we used the fact that $\overline{\bBV}$ is anti-symmetric. 
Then by H\"older's inequality, we can estimate
\begin{equation}\label{estimate_III}
\begin{aligned}
    |III| 
    &\lesssim  (R-\rho)^{-1} \|\overline \bBV\|_{s,R} \left( \int_{B_{(R+\rho)/2} \setminus B_\rho} |\nabla \Bu|^2 \right)^{\frac{1}{2}} \left( \int_{B_{(R+\rho)/2} \setminus B_\rho}| \Bu|^{\frac{2s}{s-2}} \right)^{\frac{s-2}{2s}}\\
    &\lesssim  (R-\rho)^{-1} \|\overline \bBV\|_{s,R} \left( \int_{B_{(R+\rho)/2} \setminus B_\rho} |\nabla \Bu|^2 \right)^{\frac{1}{2}} \| \Bu \eta_1^{2}\|_{\frac{2s}{s-2}, R}.
\end{aligned}    
\end{equation}
Similarly for $IV$, we have
\begin{align*}
    IV =&  \int_{B_{(R+\rho)/2} \setminus B_\rho} (\partial_k \overline{V}_{ki}) (\partial_i u_j) w_j \, dx \\
    =& - \int_{B_{(R+\rho)/2} \setminus B_\rho} \overline{V}_{ki} \big[ (\partial_{ik} u_j) w_j + (\partial_i u_j) ( \partial_k w_j)\big]\\ 
    =& -\int_{B_{(R+\rho)/2} \setminus B_\rho} \overline{V}_{ki} (\partial_i u_j) ( \partial_k w_j).
\end{align*}
Then by H\"older's inequality, \eqref{bogo_estimate} with $q = \frac{2s}{s-2}$, we can estimate
\begin{equation}\label{estimate_IV}
\begin{aligned}
    & |IV| \\
    &\lesssim  
    \rho\, (R-\rho)^{-2} \|\overline \bBV\|_{s,R} \left( \int_{B_{(R+\rho)/2} \setminus B_\rho} |\nabla \Bu|^2 \right)^{\frac{1}{2}} \left( \int_{B_{(R+\rho)/2} \setminus B_\rho}| \Bu|^{\frac{2s}{s-2}} \right)^{\frac{s-2}{2s}}\\
    &\lesssim  
    \rho\, (R-\rho)^{-2} \|\overline \bBV\|_{s,R} \left( \int_{B_{(R+\rho)/2} \setminus B_\rho} |\nabla \Bu|^2 \right)^{\frac{1}{2}} \| \Bu \eta_1^{2}\|_{\frac{2s}{s-2}, R},
\end{aligned}    
\end{equation}
where $\eta_1$ is the cut-off function defined in Lemma \ref{Lemma_u_L2}. 
Now putting \eqref{estimate_III},  \eqref{estimate_IV} together, one can obtain
\begin{equation}\label{estimate34}
\begin{aligned}
    |III|+|IV|
    \lesssim  
    R\, (R-\rho)^{-2} \|\overline \bBV\|_{s,R} \left( \int_{B_{(R+\rho)/2} \setminus B_\rho} |\nabla \Bu|^2 \right)^{\frac{1}{2}} \| \Bu \eta_1^{2}\|_{\frac{2s}{s-2}, R}.
\end{aligned}    
\end{equation}

To estimate $\| \Bu \eta_1^{2}\|_{\frac{2s}{s-2}, R}$, we use the fact that $s \in [3,6]$, and interpolation inequality to obtain
\begin{align}
    \label{inter}
        \|\Bu \eta_1^2 \| _{\frac{2s}{s-2},R}
        &\leq
        \|\Bu \eta_1^2 \|_{\frac{4s}{s+2},R}^{\frac{4s-12}{s+6}}
        \,
        \|\Bu \eta_1^2 \|_{6,R}^{\frac{18-3s}{s+6}}.
\end{align}
By Sobolev inequality and \eqref{u_L2_estimate}, we have
\begin{align*}
    \begin{aligned}
        \|\Bu \eta_1^2 \|_{6,R}
        &\lesssim
        \|\nabla (\Bu \eta_1^2)\|_{2,R} 
        \\
        &\lesssim
        \|\nabla \Bu \, \eta_1^2 \|_{2,R}
        +  \|\Bu \otimes \nabla (\eta_1^2)\|_{2,R}
        \\
        &\lesssim
        \|\nabla \Bu\|_{2,R}
        + (R-\rho)^{-1} \|\Bu \eta_1\|_{2,R}
        \\
        &\lesssim
        \|\nabla \Bu\|_{2,R}
        + (R-\rho)^{-1} 
        \left( 
        R^{\frac{3(s-2)}{4s}} 
        \|\overline{\bBV}\|_{s,R}^{\frac12}
        \|\nabla \Bu\|_{2,R}^{\frac12}
        + (R-\rho)^{-1}
        R^{\frac{3(s-2)}{2s}}  
        \|\overline{\bBV}\|_{s,R}
        \right).
    \end{aligned}
\end{align*}
This together with \eqref{u_Ls_estimate}, \eqref{inter} implies
\begin{align}\label{nonlinear_estimate}
    \begin{aligned}
        \| \Bu \eta_1^2 \|_{\frac{2s}{s-2},R} 
        &\lesssim 
        \|\overline{\bBV}\|_{s,R}^{\frac{2s-6}{s+6}}
        \|\nabla \Bu\|_{2,R} ^{\frac{12-s}{s+6}}
        \\
        &+
        (R-\rho)^{-\frac{18-3s}{s+6}}
        R^{\frac{9(s-2)(6-s)}{4s(s+6)}}
        \|\overline{ \bBV}\|_{s,R}^{\frac12} 
        \|\nabla \Bu\|_{2,R}^{\frac12}
        \\
        &+
        (R-\rho)^{-\frac{6(6-s)}{s+6}}
        R^{\frac{9(s-2)(6-s)}{2s(s+6)}}
        \|\overline{\bBV}\|_{s,R}^{\frac{12-s}{s+6}}
        \|\nabla \Bu\|_{2,R}^{\frac{2s-6}{s+6}}
        \\
        &+ 
        (R-\rho)^{-\frac{4s-12}{s+6}}
        R^{\frac{3(s-2)(s-3)}{s(s+6)}}
        \|\overline{\bBV}\|_{s,R}^{\frac{4s-12}{s+6}}
        \|\nabla \Bu\|_{2,R}^{\frac{18-3s}{s+6}}
        \\
        &+
        (R-\rho)^{-1} R^{\frac{3(s-2)}{4s}}
        \|\overline{\bBV}\|_{s,R}^{\frac{5s-6}{2(s+6)}}
        \|\nabla \Bu\|_{2,R}^{\frac{18-3s}{2(s+6)}}
        \\
        &+
        (R-\rho)^{-\frac{2(12-s)}{s+6}}
        R^{\frac{3(12-s)(s-2)}{2s(s+6)}}
        \|\overline{\bBV}\|_{s,R}.
    \end{aligned}
\end{align}
Combining \eqref{integal_split},\eqref{estimate_I+II}, \eqref{estimate34}, and \eqref{nonlinear_estimate}, and using the condition \eqref{V_condition}, we have
\begin{align}
         \label{energy2}
         \begin{aligned}
         \int |\nabla \Bu|^2 \eta_2 \, dx
          \lesssim & 
          R\,(R- \rho)^{-2} \left( \int_{B_{(R+\rho)/2} \setminus B_\rho} |\nabla \Bu|^2 \right)^{\frac{1}{2}}  \times 
          \Bigg[ R(\log R)^{ \frac{s-3}{s+6} }
        \|\nabla \Bu\|_{2,R} ^{\frac{12-s}{s+6}}           
        \\
        &
        +  R^{\frac{5s-6}{6s}} \Big(\frac{R}{R-\rho} \Big) (\log R)^{ \frac{s-3}{3s} } +   R^{\frac{11s-6}{12s}}(\log R)^{ \frac{s-3}{6s} } \|\nabla \Bu\|_{2,R}^{\frac{1}{2}} 
        \\
        &+
        R^{\frac{5s-6}{4s}}\Big(\frac{R}{R-\rho} \Big)^{\frac{18-3s}{s+6}}
        (\log R)^{ \frac{s-3}{2s} }
        \|\nabla \Bu\|_{2,R}^{\frac12}
        \\
        &+R^{\frac{3s-6}{2s}}
        \Big(\frac{R}{R-\rho} \Big)^{\frac{6(6-s)}{s+6}}
        (\log R)^{ \frac{6s-18}{s(s+6)} }
        \|\nabla \Bu\|_{2,R}^{\frac{2s-6}{s+6}}
        \\
        &+ R^{\frac{2s+3}{3s}}
        \Big(\frac{R}{R-\rho} \Big)^{\frac{4s-12}{s+6}}  
        (\log R)^{ \frac{(5s-6)(s-3)}{3s(s+6)}  }
        \|\nabla \Bu\|_{2,R}^{\frac{18-3s}{s+6}}
        \\
        &+ R^{\frac{11s-6}{12s}}
        \Big(\frac{R}{R-\rho} \Big) 
        (\log R)^{  \frac{(7s+6)(s-3)}{6s(s+6)}  }
        \|\nabla \Bu\|_{2,R}^{\frac{18-3s}{2(s+6)}}
        \\
        &+ R^{\frac{7s-12}{6s}}
        \Big(\frac{R}{R-\rho} \Big)^{\frac{2(12-s)}{s+6}}        
        (\log R)^{ \frac{2s-6}{3s} }
          \Bigg]. 
         \end{aligned}
     \end{align}
We point out that the first term $ R(\log R)^{ \frac{s-3}{s+6} }
        \|\nabla \Bu\|_{2,R} ^{\frac{12-s}{s+6}}$ is the most crucial one among those in the bracket of \eqref{energy2}, in the sense that the exponents of $R$ and $\|\nabla \Bu\|_{2,R}$ are the highest when $s \in [3,6]$. 
        
Next, we enlarge the right-hand side of \eqref{energy2} by raising the exponent of $R$ for each term in the bracket to $1$, and raising the exponent of $\frac{R}{R-\rho}$ to $9$, while changing the exponent of $\log R$ to $\frac{s-3}{s+6}$. We can always do so since the exponent of $R$ for each term is less or equal to $1$, and when it is equal to $1$, the corresponding exponent of $\log R$ is equal to $\frac{s-3}{s+6}$. Therefore, we obtain
\begin{align*}
 \int |\nabla \Bu|^2 \eta_2 \, dx \lesssim &  R^{10}(R- \rho)^{-10} (\log R)^{\frac{s-3}{s+6}} \left( \int_{B_{(R+\rho)/2} \setminus B_\rho} |\nabla \Bu|^2 \right)^{\frac{1}{2}} \times\\
 & \times \left( \| \nabla \Bu \|_{2,R}^{\frac{12-s}{s+6}}+ \|\nabla \Bu\|_{2,R}^{\frac{1}{2}} + \|\nabla \Bu\|_{2,R}^{\frac{2s - 6}{s+6}} +\|\nabla \Bu\|_{2,R}^{\frac{18-3s}{s+6}} + \|\nabla \Bu\|_{2,R}^{\frac{18-3s}{2(s+6)}} +1 \right). 
\end{align*}
Then \eqref{energy} follows from Young's inequality.
\end{proof}

In the following, we use the energy inequality obtained above to derive an upper bound of $\|\nabla \Bu\|_{2,R}$ through an iteration argument.

\begin{lemma}\label{Lemma_energy_upperbound}
    Assume that $\bBV$ satisfies \eqref{V_condition} for $s \in (3,6]$, then
    \begin{equation}
        \label{grad_u_L2_upperbound}
        \int_{B_R} |\nabla \Bu|^2 \lesssim \log R \quad \forall R > 2.
    \end{equation}
\end{lemma}

\begin{proof}
    If $\|\nabla \Bu\|_{2,R} \le 1$ for all $R >2$, then the proof is complete. Therefore, we may assume that $\|\nabla \Bu\|_{2,R} > 1$ when $R > R_0$ for some positive $R_0$. Fix $R_0 < \rho < R$.  From \eqref{energy}, we have
    \begin{equation}
        \label{iterative_relation}
        \begin{aligned}
        \|\nabla \Bu\|_{2,\rho}^2 \lesssim& R^{10}(R- \rho)^{-10} (\log R)^{\frac{s-3}{s+6}} \|\nabla \Bu\|_{2,R}^{\frac{18}{s+6}}\\
        \le& \Big(\frac{1}{2}\Big)^{\frac{11(s+6)}{s-3}} \|\nabla \Bu\|_{2,R}^2 +C R^{\frac{10(s+6)}{s-3}}(R-\rho)^{-\frac{10(s+6)}{s-3}}\log R,
    \end{aligned}
    \end{equation}
    where in the second inequality, we used the fact that $s > 3$ and Young's inequality to raise the power of $\|\nabla \Bu\|_{2,R}$. Then the lemma follows from a standard iteration argument, see for example, \cite{Gia}*{Lemma 3.1}. Indeed, let $t_0 = \rho$, $t_{i+1} - t_i = 2^{-i-1}(R-\rho)$, then $t_i \to R$ as $i \to \infty$. Using \eqref{iterative_relation} with $\rho, R$ replaced by $t_i, t_{i+1}$, and iterating $k$ times, we have
    \begin{align*}
        \|\nabla \Bu\|_{2,t_0}^2 \le& \Big(\frac{1}{2}\Big)^{\frac{11(s+6)k}{s-3}} \|\nabla \Bu\|_{2,t_k}^2 + C\sum_{i=0}^{k-1} \Big( \frac{1}{2} \Big)^{\frac{11(s+6)i}{s-3}} 2^{\frac{10(s+6)(i+1)}{s-3}} R^{\frac{10(s+6)}{s-3}}(R-\rho)^{-\frac{10(s+6)}{s-3}}\log R\\
        \le&  \Big(\frac{1}{2}\Big)^{\frac{11(s+6)k}{s-3}} \|\nabla \Bu\|_{2,R}^2 + C R^{\frac{10(s+6)}{s-3}}(R-\rho)^{-\frac{10(s+6)}{s-3}}\log R.
    \end{align*}
    Taking $k \to \infty$ and choosing $\rho = R/2$, we have completed the proof.
\end{proof}

Now we are ready to prove Theorem \ref{Thm_1} through a Saint-Venant type argument.

\begin{proof}[Proof of Theorem \ref{Thm_1}]
    Let $\eta$ be a cut-off function given by
    \begin{align*}
        \eta(r)=
        \begin{cases}
            1 & \text{if }r<1,
            \\
            -r+2 & \text{if }1\leq r \leq 2,
            \\
            0 & \text{if }r>2,
        \end{cases}
    \end{align*}
and let
$\varphi_R(x)= \eta(|x|/R).$ We define
$$
g(R):= \int_{\bR^3} |\nabla \Bu|^2 \varphi_R\, dx.
$$
Then
$$
g'(R) = \int_{B_{2R}\setminus B_R }|\nabla \Bu|^2 \eta'\big( \frac{|x|}{R} \big) \cdot \big(- \frac{|x|}{R^2} \big)\, dx \gtrsim R^{-1} \int_{B_{2R}\setminus B_R }|\nabla \Bu|^2 \, dx.
$$
Replacing $\rho,R$ by $R, 3R$ and choosing $\eta_2 = \varphi_R$ in \eqref{energy}, we have
\begin{align*}
    g(R) \lesssim& R^{\frac{1}{2}} (\log R)^{\frac{s-3}{s+6}}  g'(R)^{\frac{1}{2}} (g(3R)^{\frac{12-s}{2(s+6)}} + 1) \\
     \lesssim& \Big( R \log R g'(R) \Big)^{\frac{1}{2}},
\end{align*}
where we used \eqref{grad_u_L2_upperbound} to estimate $g(3R)$.
Assume that $\int_{\bR^3} |\nabla \Bu|^2 \neq 0$, then there exists $R_0 > 1$, such that $g(R)>0$ for all $R \ge R_0$. We have
$$
\frac{1}{R\log R} \lesssim \frac{g'(R)}{g(R)^2} \quad \forall R \ge R_0.
$$
Integrating both sides on $(R_0, R_1)$, we have
$$
\int_{R_0}^{R_1}\frac{1}{R\log R} \, dR \lesssim \frac{1}{g(R_0)} - \frac{1}{g(R_1)}.
$$
The left-hand side goes to infinity as $R_1 \to \infty$, while the right-hand side stays bounded. Therefore, we have a contradiction. Hence
$$
\int_{\bR^3} |\nabla \Bu|^2 = 0,
$$
which implies $\Bu$ is a constant, denoted by $\Bu_0$. By \eqref{u_L2_estimate} with $\rho=R/2$, we have
$$|\Bu_0| R^{\frac32} \lesssim \| \Bu \|_{2,R/2} \lesssim  R^{\frac{s-6}{2s}}\| \overline \bBV \|_{s,R} \lesssim R^{\frac{5s-6}{6s}}(\log R)^{\frac{s-3}{3s}} \quad \forall R > 2.$$
Therefore, $\Bu_0 = 0$. It finishes the proof of Theorem \ref{Thm_1}.
\end{proof}

\section{Proof of Theorem \ref{Thm_2}}\label{2thm}

In this section, we prove Theorem \ref{Thm_2}.

\begin{proof}[Proof of Theorem \ref{Thm_2}]

First note that the head pressure $Q$ satisfies
    \begin{align} \label{hpeq}
        -\Delta Q + \Bu \cdot \nabla Q = - |\boldsymbol{\omega}|^2
    \end{align}
where $\boldsymbol{\omega}=\curl \Bu$. Indeed, multiplying the momentum equation of \eqref{SNS} by $\Bu$ leads to
    \begin{align}\label{equation_u_and_Q}
        |\nabla \Bu|^2 
        =\frac{1}{2} \Delta (|\Bu|^2)
        -(\Bu\cdot \nabla )Q.
    \end{align}
Taking divergence on the momentum equation of \eqref{SNS}, we have the pressure equation
\begin{equation}\label{equation_pressure}
    -\Delta p = \div [\Bu \cdot \nabla \Bu].
\end{equation}
Then \eqref{hpeq} follows from adding \eqref{equation_u_and_Q} and \eqref{equation_pressure}.
Therefore, $Q$ satisfies the one-sided strong maximum principle. Hence it follows that $Q\equiv 0$ or $Q<0$ in $\mathbb{R}^3$. When $Q\equiv 0$, one can get $\boldsymbol{\omega} \equiv 0$ from \eqref{hpeq}, which implies 
$$-\Delta \Bu = \curl \Bo - \nabla (\div \Bu) =0 \quad \mbox{in}~~\mathbb{R}^3$$
because of the divergence free condition $\div \Bu=0$. Then the conditions \eqref{FE}, \eqref{BC} imply that $\Bu\equiv0$. Therefore, we only need to consider the case $Q<0$ in $\mathbb{R}^3$. We assume that $\Bu \not\equiv 0$ and prove by contradiction.

Define a cut-off function $\eta$ by
    \begin{align*}
        \eta(s)
        =
        \begin{cases}
            0 & \text{if }s<1/2,
            \\
            1 & \text{if }s\geq 1,
        \end{cases}
        \quad \text{and} \quad  \eta'(s)\geq 0 \quad \text{for all }s \in (0, \infty).
    \end{align*}
Then
\begin{align*}
\supp ( \eta'(-Q/\lambda) )
\subset 
\{x\in \mathbb{R}^3: -\lambda < Q(x) < -\lambda/2\}.
\end{align*}

We apply Sard's theorem to deduce that for any $\lambda \in (0, 2 \|Q\|_{\infty})$ except for possibly a set of measure zero, the level set $\{x\in \mathbb{R}^3 : Q(x)=-\lambda/2\}$ is a finite union of disjoint smooth closed surfaces
which is bounded due to \eqref{Qinfty}. We call such $\lambda$ a regular value.
Using these facts, one can apply the divergence theorem to obtain that for any regular value $\lambda$, 
    \begin{align} \label{estimate_nonlinear}
    \begin{aligned}
        \int_{\{Q<-\lambda/2\}}
        (\Bu \cdot \nabla )Q
        \thinspace
        \eta 
        \left( 
        -\frac{Q}{\lambda}
        \right) \thinspace dx
        &=
        \int_{ \{Q<-\lambda/2\} }
        \Bu \cdot \nabla 
        \left( 
            \int_0^{Q(x)}
            \eta \left( 
            -\frac{s}{\lambda}
            \right)
            ds
        \right)
        dx
        \\
        &=
        \int_{\{Q=-\lambda/2\}}
        \Bu \cdot \Bn 
        \int_0 ^{Q(x)}
        \eta 
        \left(
        -\frac{s}{\lambda}
        \right)
        ds
        \thinspace
        d\sigma
        \\
        &= \int_0^{-\frac{\lambda }{2}}
        \eta 
        \left(
        -\frac{s}{\lambda}
        \right)
        ds
        \int_{\{Q=-\lambda/2\}}
        \Bu \cdot \Bn \thinspace d\sigma=0
    \end{aligned}
    \end{align}
where we have use the divergence free condition $\div \Bu=0$ at the last equality. Here the set $\{Q<-\lambda/2\}$ indeed refers to $\{x\in \mathbb{R}^3: Q(x) <-\lambda/2\}$. Similar notations will be used hereafter.

Multiply \eqref{equation_u_and_Q} with $\eta(-Q/\lambda)$ and integrate by parts along with \eqref{estimate_nonlinear}
to obtain that for every regular value $\lambda>0,$
    \begin{align} \label{main_estimate}
    \begin{aligned}
        &2\int |\nabla \Bu|^2 \eta
        \left(
        -\frac{Q}{\lambda}
        \right) \thinspace dx \\
        &= 
        -\int 
        \nabla (|\Bu|^2)
        \cdot
        \nabla \left( 
        \eta \left( 
        -\frac{Q}{\lambda }
        \right)
        \right)
        dx
        \\
        &=
        -\int_{ 
        \{-\lambda < Q< -\lambda/2\}
        }
        u_i \partial_{x_j} u_i 
        \eta'
        \left(
        -\frac{Q}{\lambda }
        \right)
        \left(
        -\frac{\partial_{x_j}Q}{\lambda}
        \right) \thinspace dx
        \\
        &\leq
        \int_{\{-\lambda < Q< -\lambda/2\}}
        |\Bu||\nabla \Bu |
        \left|
        \eta' 
        \left( 
        -\frac{Q}{\lambda }
        \right)
        \right|
        \frac{|\nabla Q|}{\lambda}
        \thinspace dx
        \\
        &\leq
        \left( 
        \sup 
        _{\supp \thinspace \eta'(-Q/\lambda)}
        |\Bu|
        \right)
         \left( 
        \int |\nabla \Bu|^2
        \eta' 
        \left( 
        -\frac{Q}{\lambda }
        \right)
        \left(
        \frac{1}{\lambda}
        \right)
        dx
        \right)^{1/2}
        \left( 
        \int
        \frac{|\nabla Q|^2 }{\lambda}
         \eta' 
        \left( 
        -\frac{Q}{\lambda }
        \right)
        dx
        \right)^{1/2}.
    \end{aligned}
    \end{align}

Now we will use the head pressure equation \eqref{hpeq}.
Multiply \eqref{hpeq} by $\eta(-Q/\lambda)$ and integrate by parts along with \eqref{estimate_nonlinear} to obtain 
    \begin{align*}
        \int 
        _{ \{-\lambda < Q< -\lambda/2\}}
        \frac{|\nabla Q|^2}{\lambda }
        \eta'
        \left(
        -\frac{Q}{\lambda }
        \right)
        dx
        =
        \int
        |\Bo|^2 
        \thinspace 
        \eta\left( 
        -\frac{Q}{\lambda}
        \right)
        dx.
    \end{align*}
Applying this estimate in \eqref{main_estimate}, one can obtain
    \begin{align*}
        &2
        \int |\nabla \Bu|^2 
        \eta
        \left(
        -\frac{Q}{\lambda }
        \right)
        dx \\
        &\leq 
         \left( 
        \sup 
        _{\supp \thinspace \eta'(-Q/\lambda)}
        |\Bu|
        \right)
         \left( 
        \int |\nabla \Bu|^2
        \eta' 
        \left( 
        -\frac{Q}{\lambda }
        \right)
        \left(
        \frac{1}{\lambda}
        \right)
        dx
        \right)^{1/2}
        \left( 
        \int
       |\Bo|^2 
       \eta\left(
       -\frac{Q}{\lambda }
       \right)
        dx
        \right)^{1/2}
        \\
        &\leq
          \left( 
        \sup 
        _{\supp \thinspace \eta'(-Q/\lambda)}
        |\Bu|
        \right)
         \left( 
        \int |\nabla \Bu|^2
        \eta' 
        \left( 
        -\frac{Q}{\lambda }
        \right)
        \left(
        \frac{1}{\lambda}
        \right)
        dx
        \right)^{1/2}
        \left( 
        \int
       |\nabla \Bu|^2 
       \eta\left(
       -\frac{Q}{\lambda }
       \right)
        dx
        \right)^{1/2}
    \end{align*}
Hence
    \begin{align}
     \label{main_estimate2}
        \int |\nabla \Bu|^2 
       \eta
       \left(
       -\frac{Q}{\lambda }
       \right)
       dx
       \leq
       \frac{1}{4}
        \left( 
        \sup 
        _{\supp \thinspace \eta'(-Q/\lambda)}
        |\Bu|^2
        \right)
        \int |\nabla \Bu|^2
        \eta' 
        \left( 
        -\frac{Q}{\lambda }
        \right)
        \left(
        \frac{1}{\lambda}
        \right)
        dx
    \end{align}

Next we define the following function
    \begin{align*}
    g(\lambda)
    :=\int |\nabla \Bu|^2 \eta
    \left(
    -\frac{Q}{\lambda }
    \right)
    dx,
    \quad
    \lambda > 0.
    \end{align*}
As $\eta$ is a non-decreasing function, it follows that $g$ is a non-increasing function, and 
\begin{equation}\label{g_limit}
    \lim_{\lambda \to 0_+} g(\lambda) = \|\nabla \Bu \|_{2}
\end{equation}
due to Dominated Convergence Theorem.
Then 
    \begin{align*}
       -g'(\lambda)
        &=
        \int 
        _{
        \{-\lambda < Q< -\lambda/2\}
        }
        |\nabla \Bu|^2 
        \eta' 
        \left( 
        -\frac{Q}{\lambda }
        \right)
        \left( 
        \frac{-Q}{\lambda^2}
        \right)
        dx 
        \\
        &\geq
        \frac{1}{2}
          \int 
        _{
        \{-\lambda < Q< -\lambda/2\}
        }
        |\nabla \Bu|^2 
        \eta' 
        \left( 
        -\frac{Q}{\lambda }
        \right)
        \left( 
        \frac{1}{\lambda}
        \right)
        dx,
    \end{align*}
where we used the fact that $\eta' \ge 0$. Applying this estimate in \eqref{main_estimate2}, one can obtain
    \begin{align*}
        g(\lambda) 
        \leq
        -\frac{1}{2} g'(\lambda) 
        \left( 
        \sup 
        _{\supp \thinspace \eta'(-Q/\lambda)}
        |\Bu|^2
        \right).
    \end{align*}
This holds for every regular value $\lambda>0$. By \eqref{cond2}, there exists an $s_0 \in (0,1/2)$ such that
\begin{equation}\label{lambda_inequality}
    g(\lambda) \le -C g'(\lambda) \lambda \log \lambda.
\end{equation}
for all regular values $\lambda \in (0, s_0)$. Since we have assume $\Bu \not\equiv 0$, there exists an $s_1 \in (0, s_0)$, such that $g(\lambda) > 0$ for all regular values $\lambda \in (0, s_1)$. Dividing \eqref{lambda_inequality} by $g(\lambda)\lambda \log \lambda$ on both sides and integrating over $(\lambda_1, \lambda_2)$ for regular values $\lambda_1, \lambda_2 \in (0,s_1)$, we have
$$
\int_{\lambda_1}^{\lambda_2} \frac{1}{ \lambda \log \lambda} \, d\lambda \le - C \int_{\lambda_1}^{\lambda_2} \frac{g'(\lambda)}{g(\lambda)} \, d\lambda = C \log \left( \frac{g(\lambda_1)}{g(\lambda_2)} \right).
$$
Taking $\lambda_1 \to 0$, one can see the left-hand-side goes to infinity while the right-hand-side stays bounded due to \eqref{g_limit}. This leads to a contradiction. Therefore, $\Bu \equiv 0$.
\end{proof}

\bibliographystyle{amsplain}
\bibliography{mybib.bib}
\end{document}